\documentclass[12pt]{amsart}
\usepackage{amssymb}
\usepackage{color}
\usepackage{amsthm}
\usepackage[pdftex]{graphicx}
\usepackage[margin=1in]{geometry}  

\newcommand{\grad}{\nabla}
\newcommand{\bfu}{\mathbf{u}}

\newcommand{\bfX}{\mathbf{X}}
\newcommand{\bfY}{\mathbf{Y}}
\newcommand{\bfa}{\boldsymbol{a}}
\newcommand{\bfx}{\mathbf{x}}

\DeclareMathOperator{\diver}{div}
\DeclareMathOperator{\curl}{curl}

\newtheorem{theorem}{Theorem}[section]

\newtheorem{lemma}[theorem]{Lemma}

\begin{document}
%
%
\setlength{\parskip}{6pt plus 1pt minus 1pt}

\title[Damped Infinite Energy Solutions of the 3D Euler and Boussinesq Equations]{Damped Infinite Energy Solutions of the 3D Euler\\ and Boussinesq equations}

\author{William Chen}
\address{Department of Mathematics and Statistics, Williams College, Williamstown, MA 01267}
\email{wyc1@williams.edu}
\author{Alejandro Sarria}
\address{Department of Mathematics and Statistics, Williams College, Williamstown, MA 01267}
\email{Alejandro.Sarria@williams.edu}




\begin{abstract}
We revisit a family of infinite-energy solutions of the 3D incompressible Euler equations proposed by Gibbon et al. \cite{gibbon1} and shown to blowup in finite time by Constantin \cite{constantin2}. By adding a damping term to the momentum equation we examine how the damping coefficient can arrest this blowup. Further, we show that similar infinite-energy solutions of the inviscid 3D Boussinesq system with damping can develop a singularity in finite time as long as the damping effects are insufficient to arrest the (undamped) 3D Euler blowup in the associated damped 3D Euler system.
\end{abstract}

\maketitle

\section{Introduction}

We consider a family of exact infinite-energy solutions of two three-dimensional (3D) fluid models with a damping term, the incompressible Euler equations
\begin{equation}
\label{dampede}
\begin{cases}
&\bfu_t+\bfu\cdot\nabla \bfu+\alpha \bfu=-\nabla p,
\\
&\diver\bfu=0,
\end{cases}
\end{equation}
and the inviscid Boussinesq system 
\begin{equation}
\label{dampedb}
\begin{cases}
&\bfu_t+\bfu\cdot\nabla \bfu+\alpha \bfu=-\nabla p+\theta\mathbf{e}_3,
\\
&\theta_t+\bfu\cdot\nabla\theta=0,
\\
&\diver\bfu=0.
\end{cases}
\end{equation}
In \eqref{dampede}-\eqref{dampedb}, $\bfu$ represents the fluid velocity, $p$ is the scalar pressure, $\theta$ is the scalar temperature in the context of thermal convection or the density in the modeling of geophysical fluids, $\mathbf{e}_3=(0,0,1)^T$, and $\alpha\in\mathbb{R}^+$ is a real parameter. For $\alpha=0$, \eqref{dampede} reduces to the standard 3D Euler equations describing the motion of an ideal, incompressible homogeneous fluid, while \eqref{dampedb} becomes the standard 3D inviscid Boussinesq system modeling large scale atmospheric dynamics and oceanic flows \cite{gill, majda1, pedlosky}. If $\theta\equiv0$,   \eqref{dampedb} reduces to \eqref{dampede}. When $\alpha\bfu$ in \eqref{dampede} is replaced by the diffusion term $-\nu\Delta\bfu$, we obtain the classical 3D Navier-Stokes equations. The global (in time) regularity problem for the aforementioned 3D models are long-standing open problems in mathematical fluid dynamics. See Constantin \cite{constantin1} for a history
and survey of results on the 3D Euler regularity problem and Fefferman \cite{fefferman} for a more precise account of the Navier-Stokes regularity problem. In general, the main obstacle in obtaining global existence of smooth solutions of these 3D models for general initial data is controlling nonlinear growth due to vortex stretching \cite{hou22, hou1}. To gain insight into this challenging problem, many researchers have turned their efforts to the 2D viscous Boussinesq equations 
\begin{equation}
\label{b}
\begin{cases}
&\bfu_t+\bfu\cdot\nabla \bfu=-\nabla p+\nu\Delta\bfu+\theta\mathbf{e}_2,
\\
&\theta_t+\bfu\cdot\nabla\theta=\kappa\Delta\theta,
\\
&\diver\bfu=0.
\end{cases}
\end{equation}
System \eqref{b} can be shown to be formally identical to the 3D Euler or Navier-Stokes equations for axisymmetric swirling flows and retains key features of the 3D models such as the vortex stretching  mechanism (see, e.g., \cite{majdabertozzi}). The global regularity issue for \eqref{b} has been settled in the affirmative under various degrees of viscosity and dissipation: with full viscosity $\nu>0$ and $\kappa>0$, partial viscosity $\nu>0$ and $\kappa=0$, or $\nu=0$ and $\kappa>0$, for anisotropic models \cite{adhikaria01, chae07, hou0, lai}, and with fractional Laplacian dissipation (see \cite{yang} and references therein). In contrast, the question of global regularity for \eqref{b} in the inviscid case $\nu=\kappa=0$ remains open; it is not apparent how to control vortex stretching when there is no dissipation ($\nu=0$) and no thermal diffusion ($\kappa=0$). Using a somewhat different approach, Adhikaria et al. \cite{adhikaria} replaced $\nu\Delta\bfu$ and $\kappa\Delta\theta$ in \eqref{b} with damping terms $-\alpha\bfu$ and $-\beta\theta$ for $\alpha>0$ and $\beta>0$ real parameters. Although the resulting damping effects are insufficient to control vortex stretching for general initial data, the authors showed that a local (in time) solution will persist globally in time if the initial data is small enough in some homogeneous Besov space. 

The aim of this work is to examine how damping affects the global regularity of a particular class of infinite-energy solutions of \eqref{dampede} and \eqref{dampedb} which, in the absence of damping ($\alpha=0$), blowup in finite-time from smooth initial data. More particularly, the fluid velocity and temperature considered here have the form
\begin{equation}
\label{ansatz}
\bfu(\bfx,z,t)=(u(\bfx,t),v(\bfx,t),z\gamma(\bfx,t)),\qquad \theta(\bfx,z,t)=z\rho(\bfx,t)
\end{equation}
for $(\bfx,z)=(x,y,z)$. Our spatial domain will be the semi-bounded 3D channel 
\begin{equation}
\label{domain}
\Pi\equiv\{(\bfx,z)\in Q\times\mathbb{R}\}
\end{equation}
of rectangular periodic cross-section $Q\equiv[0,1]^2$ with $\bfu$ and $\theta$ both periodic in the $x$ and $y$ variables with period one. Note that the unbounded geometry of \eqref{domain} in the vertical direction endows the fluid under consideration with, at best, locally finite kinetic energy.

Solutions of the 3D incompressible Euler equations
\begin{equation} 
\label{e}
\begin{cases}
&\bfu_t+\bfu\cdot\nabla \bfu=-\nabla p
\\
&\diver\bfu=0
\end{cases}
\end{equation}
of the form \eqref{ansatz}i) were proposed by Gibbon et al. \cite{gibbon1}, and then shown to blowup in finite time numerically by Ohkitani and Gibbon \cite{ohkitani}, and analytically by Constantin \cite{constantin2}. See also \cite{childress, stuart, sarria1, sarria2} for blowup results of other similar infinite-energy solutions of the Euler and Boussinesq equations in two and three dimensions. For convenience of the reader, we now summarize Constantin's blowup result. 

The set-up used by Constantin is effectively the same as the one considered here. Imposing a velocity field of the form \eqref{ansatz}i) on \eqref{e} subject to periodicity in the $x$ and $y$ variables of period one and with spatial domain \eqref{domain}, it follows that the vertical component of the velocity field, $z\gamma$, satisfies the vertical component of the momentum equation \eqref{e}i) if the mean-zero function $\gamma=-\left(u_x+v_y\right)$ solves the nonlocal two-dimensional equation 
\begin{equation}
\label{undampedgammaeqn}
\gamma_t+\bfu'\cdot\nabla\gamma=-\gamma^2+2\int_{Q}{\gamma^2\,d\mathbf{x}}
\end{equation}
with $\bfu'=(u,v)$. For $(\bfa,t)\equiv(a_1,a_2,t)$, $\bfa\in Q$, Constantin constructed the solution formula   
\begin{equation}
\label{undampedsln}
\begin{split}
\gamma(\bfY(\bfa,t),t)=-\frac{\tau'(t)}{\tau(t)}\left\{\frac{1}{1+\gamma_0(\bfa)\tau(t)}-\left(\int_{Q}{\frac{d\bfa}{1+\gamma_0(\bfa)\tau(t)}}\right)^{-1}\int_{Q}{\frac{d\bfa}{(1+\gamma_0(\bfa)\tau(t))^2}}\right\}
\end{split}
\end{equation}
for $\gamma(\bfa,0)=\gamma_0(\bfa)$, $\tau$ satisfying the initial value problem (IVP)
\begin{equation}
\label{undampedtau}
\tau'(t)=\left(\int_Q{\frac{d\bfa}{1+\gamma_0(\bfa)\tau(t)}}\right)^{-2},\qquad \tau(0)=0,
\end{equation}
and $\bfa\to\mathbf{Y}(\bfa,t)$ the 2D flow-map defined by
\begin{equation*}
\label{undampedflowmap}
\frac{d\mathbf{Y}}{dt}=\bfu'(\mathbf{Y}(\bfa,t),t),\qquad \mathbf{Y}(\bfa,0)=\bfa.
\end{equation*}
By comparing the blowup rates of the time integrals in \eqref{undampedsln} against the local term, Constantin proved the following blowup result for a large class of smooth initial data $\gamma_0$. 

\begin{theorem}[\cite{constantin2}]
\label{blowupundamped}
Consider the initial boundary value problem for \eqref{undampedgammaeqn} with smooth mean-zero initial data $\gamma_0$ and periodic boundary conditions. Suppose $\gamma_0$ attains a negative minimum $m_0$ at a finite number of locations $\bfa_0\in Q$ and, near these locations, $\gamma_0$ has non-vanishing second-order derivatives. Set $\tau^*=-\frac{1}{m_0}$ and let
\begin{equation}
\label{undampedeulertime}
t^E(\tau)=\int_0^{\tau}\left(\int_{Q}{\frac{d\bfa}{1+\gamma_0(\bfa)\mu}}\right)^2d\mu.
\end{equation}
Then there exists a finite time $T^E>0$, given by 
\begin{equation*}
\label{undampedeulerblowup}
T^E\equiv\lim_{\tau\nearrow\tau^*}t^E(\tau),
\end{equation*}
such that both the maximum and minimum values of $\gamma$ diverge to positive and respectively negative infinity as $t\nearrow T^E$. 
\end{theorem}

The outline for the remainder of the paper is as follows. In Section \ref{prelim} we introduce the damped two-dimensional equations \eqref{dampedeqnb}-\eqref{dampedeqne} and summarize the main results of the paper. Then in Section \ref{solution} we derive the solution formulae \eqref{gammab}-\eqref{tauivpe}, which we use in Section \ref{proofs} to prove the Theorems.

\section{Preliminaries}
\label{prelim}


\subsection{The Damped Two-dimensional Equations}\hfill

As stated in the previous section, we are interested in the global regularity of solutions of \eqref{dampede} and \eqref{dampedb} of the form \eqref{ansatz} subject to periodic boundary conditions in the $x$ and $y$ variables (period one), and with spatial domain \eqref{domain}. First note that, from incompressibility and periodicity, $\gamma=-\left(u_x+v_y\right)$ satisfies the mean-zero condition
\begin{equation}
\label{zeromean}
\int_Q{\gamma(\bfx,t)\,d\bfx}=0.
\end{equation}
Then imposing the ansatz \eqref{ansatz} on the damped Boussinesq system \eqref{dampedb}, it is easy to check that the vertical component of the velocity field, $z\gamma$, and the scalar temperature, $z\rho$, satisfy the vertical component of \eqref{dampedb}i) and equation \eqref{dampedb}ii) if $\gamma$ and $\rho$ solve the nonlocal 2D system
\begin{equation}
\label{dampedeqnb}
\begin{cases}
\gamma_t+\bfu'\cdot\nabla\gamma=\rho-\gamma^2-\alpha\gamma+I(t),\quad &\bfx\in Q,\quad t>0,
\\
\rho_t+\bfu'\cdot\nabla\rho=-\gamma\rho,\qquad &\bfx\in Q,\quad t>0
\end{cases}
\end{equation}
with $\bfu'=(u,v)$, $\alpha>0$ a real parameter, and
\begin{equation}
\label{nonlocalb}
I(t)=2\int_Q{\gamma^2(\bfx,t)\,d\bfx}-\int_Q\rho(\bfx,t)\,d\bfx.
\end{equation}

For $\rho\equiv0$, \eqref{dampedeqnb}-\eqref{nonlocalb} reduces to 
\begin{equation}
\label{dampedeqne}
\begin{cases}
\gamma_t+\bfu'\cdot\nabla\gamma=-\gamma^2-\alpha\gamma+I(t),\quad &\bfx\in Q,\quad t>0,
\\
I(t)=2\int_Q{\gamma^2(\bfx,t)\,d\bfx},
\end{cases}
\end{equation}
which is just the associated 2D equation obtained from the vertical component of the damped 3D Euler system \eqref{dampede}. 

For simplicity we will refer to equations \eqref{undampedgammaeqn} and \eqref{dampedeqne}, and the system \eqref{dampedeqnb}-\eqref{nonlocalb}, as \emph{the undamped Euler equation}, \emph{the damped Euler equation}, and \emph{the damped Boussinesq system}, respectively. 

Before summarizing the main results of the paper, we define some notation that will be helpful in differentiating among solutions of the various equations under consideration.

\subsection{Notation}\hfill

For $\alpha>0$, we denote by $(\gamma^{B}_{\alpha},\rho_{\alpha})$ and $\gamma^{E}_{\alpha}$ the solution of the damped Boussinesq system \eqref{dampedeqnb}-\eqref{nonlocalb} and the damped Euler equation \eqref{dampedeqne}, respectively. The undamped ($\alpha=0$) counterparts of $(\gamma^{B}_{\alpha},\rho_{\alpha})$ and $\gamma^E_{\alpha}$ will be denoted by dropping the subscript, i.e., $(\gamma^{B},\rho)$ and respectively $\gamma^E$, with the latter given (along characteristics) by formula \eqref{undampedsln}. Other notation will be introduced in later sections in a similar manner. Lastly, by $C$ we mean a generic positive constant that may change in value from line to line. 



\subsection{Summary of Results}\hfill


For a smooth initial condition $\gamma_0$ satisfying the conditions of Theorem \ref{blowupundamped}, we determine in Theorem \ref{dampedeulerthm} below ``how much'' damping is required for the solution $\gamma^E_{\alpha}$ of the damped Euler equation \eqref{dampedeqne} to persist globally in time, or alternatively, for the finite-time blowup of the solution $\gamma^E$ of the undamped Euler equation \eqref{undampedgammaeqn} to be suppressed. 

\begin{theorem}
\label{dampedeulerthm}
Consider the damped Euler equation \eqref{dampedeqne} with smooth mean-zero initial data $\gamma_0$ and periodic boundary conditions. Suppose $\gamma_0$ satisfies the conditions of Theorem \ref{blowupundamped}, so the solution $\gamma^E$ of the undamped Euler equation \eqref{undampedgammaeqn} blows up at a finite time $T^E$. Then for $\alpha\geq1/T^E$, the solution $\gamma^E_{\alpha}$ of the damped Euler equation \eqref{dampedeqne} exists globally in time. More particularly, for $\alpha=1/T^E$, $\gamma^E_{\alpha}$ converges to a non-trivial steady state as $t\to+\infty$, whereas, for  $\alpha>1/T^E$, convergence is to the trivial steady state. In contrast, if $0<\alpha<1/T^E$, then there exists a finite time $T_{\alpha}^E>T^E$ such that the maximum and minimum values of $\gamma^E_{\alpha}$ diverge to positive and respectively negative infinity as $t\nearrow T_{\alpha}^E$. More particularly, let $\tau^*=-1/m_0$ for $m_0$ the negative minimum of $\gamma_0$ attained at a finite number of points in $Q$, and set $t_{\alpha}^E(\tau)=-\frac{1}{\alpha}\ln\left|1-\alpha t^E(\tau)\right|$ for $t^E(\tau)$ the undamped Euler time variable  \eqref{undampedeulertime}. Then for $0<\alpha<1/T^E$, the finite blowup time is $T_{\alpha}^E\equiv\lim_{\tau\nearrow\tau^*}t_{\alpha}^E(\tau)$.
\end{theorem}

Before summarizing our next result, we note that Gibbon and Ohkitani \cite{gibbon3} established a regularity criterion of BKM \cite{BKM} type whereby a solution $\gamma^E$ of \eqref{undampedgammaeqn} blows up at a finite time $T^E$ if and only if
\begin{equation}
\label{bkm}
\int_0^{T^E}{\|\gamma^E(\bfx,s)\|_{\infty}}ds=+\infty
\end{equation}
for $\|\cdot\|_{\infty}$ the supremum norm. Let $(\bfu, p)$ be a solution of the damped Euler system \eqref{dampede} for $\bfu$ as in \eqref{ansatz}i). Since the vertical component $\omega=v_x-u_y$ of the vorticity $\boldsymbol{\omega}=\curl\bfu$ satisfies
\begin{equation*}
\label{dampedvorticityeqn}
\omega_t+\bfu'\cdot\grad\omega=(\gamma_{\alpha}^E-\alpha)\omega,
\end{equation*}
we may refer to $\gamma_{\alpha}^E$ as the vorticity stretching rate. In the spirit of \eqref{bkm}, a similar BKM-type criterion can be established for a solution of the damped Euler equation \eqref{dampedeqne} where the blowup
time $T_{\alpha}^E$ is defined as the smallest time at which
$$\int_0^{T_{\alpha}^E}{\|\gamma_{\alpha}^E(\bfx,s)\|_{\infty}}ds=+\infty.$$
Additionally, the argument used to prove Theorem 1.1 in \cite{sarria2}, together with Theorem  \ref{dampedboussinesqthm} below, shows that the same regularity criterion\footnote[1]{With $\gamma_{\alpha}^E$ replaced by $\gamma_{\alpha}^B$} is true for a solution of the damped Boussinesq system \eqref{dampedeqnb}-\eqref{nonlocalb}. Part one of our next Theorem shows that for smooth $\gamma_0$ satisfying the conditions of Theorem \ref{blowupundamped} and $\rho_0\geq0$, the existence of a finite blowup time $T^E$ for the undamped Euler equation \eqref{undampedgammaeqn} leads to finite-time blowup of the damped Boussinesq system \eqref{dampedeqnb}-\eqref{nonlocalb} if the damping coefficient satisfies $0<\alpha<1/T^E$ and there is at least one point $\bfa_1\in Q$ such that $\gamma_0(\bfa_1)=m_0$ and $\rho_0(\bfa_1)=0$. In particular, we find that the time integral of $\gamma_{\alpha}^B$ diverges to \emph{negative infinity} for $\bfa=\bfa_1$, in agreement with the aforementioned BKM-type criterion. The second part of Theorem \ref{dampedboussinesqthm} shows that finite-time blowup is not restricted to nonnegative $\rho_0$, although the blowup mechanism we uncover for nonpositive $\rho_0$ is of a different nature and the singularity is effectively one dimensional. Briefly, we show the existence of smooth $\gamma_0$ attaining its negative minimum $m_0$ at infinitely many points $\bfa'\in Q$, and $\rho_0\leq0$ satisfying $\rho_0(\bfa')\neq0$, such that for $T^B>0$ the finite blowup time of the associated undamped Boussinesq system (i.e. \eqref{dampedeqnb}-\eqref{nonlocalb} with $\alpha=0$), if $0<\alpha<1/T^B$, then the time integral of $\gamma_{\alpha}^B$ diverges to \emph{positive infinity} for $\bfa\neq\bfa'$. 


\begin{theorem}
\label{dampedboussinesqthm}
Consider the damped Bousinesq system \eqref{dampedeqnb}-\eqref{nonlocalb} with periodic boundary conditions.

\begin{enumerate}

\item Suppose $\gamma_0$ satisfies the conditions of Theorem \ref{blowupundamped}, so the solution $\gamma^E$ of the undamped Euler equation \eqref{undampedgammaeqn} blows up at a finite time $T^E$. Further, let $0<\alpha<1/T^E$, so that by Theorem \ref{dampedeulerthm} the solution $\gamma_{\alpha}^E$ of the damped Euler equation \eqref{dampedeqne} blows up at a finite time $T_{\alpha}^E$. Assume $\gamma_0$ attains its negative minimum $m_0$ at finitely many points $\bfa_i\in Q$, $1\leq i\leq n$, and $\rho_0\geq0$ is smooth with $\rho_0(\bfa_j)=0$ for some $1\leq j\leq n$. Then there exists a finite time $T_{\alpha}^B$, satisfying $0<T_{\alpha}^B<T_{\alpha}^E$, such that 
$$J(\bfa_j,t)= \operatorname{exp}\,\left\{-\int_0^t{\gamma_{\alpha}^B(\bfX(\bfa_j,s),s)ds}\right\}\to+\infty$$
as $t\nearrow T_{\alpha}^B$ for $J=\det\left\{\frac{\partial\bfX}{\partial\bfa}\right\}$ the Jacobian of the 2D flow-map (see \eqref{dampedflowmap}).  

\vspace{0.05in}

\item There exist smooth mean-zero initial data $\gamma_0$ attaining its negative minimum $m_0$ at infinitely many points $\bfa'\in Q$, and smooth $\rho_0\leq0$ with $\rho_0(\bfa')\neq0$, such that for $T^B>0$ the finite blowup time of the associated undamped Boussinesq system, if $0<\alpha<1/T^B$ and $\bfa\neq\bfa'$, then 
$$J(\bfa,t)= \operatorname{exp}\,\left\{-\int_0^t{\gamma_{\alpha}^B(\bfX(\bfa,s),s)ds}\right\}\to0$$
as $t\nearrow T_{\alpha}^B=-\frac{1}{\alpha}\ln\left(1-\alpha T^B\right)>T^B$. 
\end{enumerate}

\end{theorem}


%

\section{Solution along characteristics}
\label{solution}

In this section we derive the representation formula \eqref{gammab}-\eqref{tauivpb} for the general solution $(\gamma_{\alpha}^B(\bfx,t),\rho_{\alpha}(\bfx,t))$ (along characteristics) of the damped Boussinesq system  \eqref{dampedeqnb}-\eqref{nonlocalb}. Our approach is a natural generalization to higher dimensions of the arguments in \cite{sarria1, sarria2}, and is somewhat more direct than the argument used in \cite{constantin2} to derive \eqref{undampedsln}-\eqref{undampedtau}. 

For $(\bfa,t)\equiv(a_1,a_2,t)$, $\bfa\in Q$, and $\bfu'=(u,v)$, define the 2D flow-map $\bfa\to\bfX(\bfa,t)$ as the solution of the IVP
\begin{equation}
\label{dampedflowmap}
\frac{d\mathbf{X}}{dt}=\bfu'(\bfX(\bfa,t),t),\qquad \mathbf{X}(\bfa,0)=\bfa.
\end{equation}
Integrating equation \eqref{dampedeqnb}ii) along the flow-map yields
\begin{equation}
\label{eq1}
\rho_{\alpha}(\bfX(\bfa,t),t)=\rho_0(\bfa)\operatorname{exp}\,\left\{-\int_0^t{\gamma_{\alpha}^B(\bfX(\bfa,s),s)\,ds}\right\}.
\end{equation}
Note that 
\begin{equation}
\label{rho}
\rho_{\alpha}(\bfX(\bfa,t),t)=\rho_0(\bfa)J(\bfa,t)
\end{equation}
for $J(\bfa,t)=\det\left\{\frac{\partial\bfX}{\partial\bfa}\right\}$ the Jacobian determinant of $\bfX$ satisfying
\begin{equation}
\label{jac}
J(\bfa,t)=\text{exp}\,\left\{-\int_0^t{\gamma_{\alpha}^B(\bfX(\bfa,s),s)\,ds}\right\}.
\end{equation}
Formula \eqref{rho} follows directly from \eqref{eq1} and the fact that the Jacobian satisfies
\begin{equation}
\label{jacid}
J_t(\bfa,t)=-J(\bfa,t)\gamma_{\alpha}^B(\bfX(\bfa,t),t),\qquad J(\bfa,0)\equiv1.
\end{equation}
Next we use the above formulas to derive a second-order linear ODE for $\xi=J^{-1}$. 

From \eqref{dampedeqnb}i), \eqref{rho}, and \eqref{jacid}, it follows that
\begin{equation}
\label{eq2}
\begin{split}
\frac{\partial}{\partial t}(\gamma_{\alpha}^B(\bfX(\bfa,t),t))=\rho_0(\bfa)J(\bfa,t)-\left(\frac{J_t(\bfa,t)}{J(\bfa,t)}\right)^2+\alpha\frac{J_t(\bfa,t)}{J(\bfa,t)}+I(t).
\end{split}
\end{equation}
Then differentiating \eqref{jacid} with respect to $t$, and using \eqref{jacid} and \eqref{eq2} on the resulting expression leads to
\begin{equation}
\label{eq3}
\begin{split}
J_{tt}&=\left(2\frac{J_t^2}{J^2}-\rho_0J-\alpha\frac{J_t}{J}-I(t)\right)J.
\end{split}
\end{equation}
Setting $\xi=J^{-1}$ in \eqref{eq3} yields, after some rearranging, 
\begin{equation}
\label{ode}
\xi_{tt}(\bfa,t)+\alpha\xi_t(\bfa,t)-I(t)\xi(\bfa,t)=\rho_0(\bfa),
\end{equation}
a second-order linear ODE parametrized by $\bfa\in Q$. Fix $\bfa\in Q$ and set $f(t)=\xi(\bfa,t)$. Then by \eqref{jacid} and \eqref{ode}, $f$ solves the IVP
\begin{equation}
\label{ode2}
f''(t)+\alpha f'(t)-I(t)f(t)=\rho_0,\qquad f(0)=1,\,\,\, f'(0)=\gamma_0(\bfa)
\end{equation}
for ${}^{'}\equiv\frac{d}{dt}$. To find a general solution of \eqref{ode2} we follow a standard variation of parameters argument. Consider the associated homogeneous equation
\begin{equation}
\label{homo}
f_h''(t)+\alpha f_h'(t)-I(t)f_h(t)=0.
\end{equation}
Let $\phi_1(t)$ and $\phi_2(t)$ be two linearly independent solutions of \eqref{homo} satisfying $\phi_1(0)=\phi_2'(0)=1$ and $\phi_1'(0)=\phi_2(0)=0$. Then by reduction of order, the general solution of \eqref{homo} takes the form $f_h(t)=\phi_1(t)(c_1(\bfa)+c_2(\bfa)\tau_{\alpha}(t))$ for
\begin{equation*}
\label{tau}
\tau_{\alpha}(t)=\int_0^t{\frac{e^{-\alpha s}}{\phi_1^2(s)}\,ds}.
\end{equation*}
Now consider a particular solution  $f_p(t)=v_1(\bfa,t)\phi_1(t)+v_2(\bfa,t)\phi_2(t)$ of \eqref{ode2} for $\phi_2(t)=\tau_{\alpha}(t)\phi_1(t)$ and $v_1$ and $v_2$ to be determined. After some standard computations, we can write the general solution of \eqref{ode2} as
\begin{equation}
\label{particular}
\xi(\bfa,t)=\phi_1(t)\left[\Psi(\bfa,t)-\rho_0(\bfa)\sigma(t)\right]
\end{equation}
for 
\begin{equation}
\label{psi}
\begin{split}
\Psi(\bfa,t)=1+\gamma_0(\bfa)\tau_{\alpha}(t)
\end{split}
\end{equation}
and
\begin{equation}
\label{sigma}
\begin{split}
\sigma(t)&=\int_0^t{e^{\alpha s}\tau_{\alpha}(s)\phi_1(s)ds}-\tau_{\alpha}(t)\int_0^t{e^{\alpha s}\phi_1(s)ds}
\\
&=
-\int_0^t\int_s^t{\frac{\phi_1(s)}{\phi_1^2(z)}e^{\alpha(s-z)}dzds}.
\end{split}
\end{equation}
The Jacobian is now obtained from \eqref{particular} and $\xi=J^{-1}$ as
\begin{equation}
\label{dampedjacobianb0}
J(\bfa,t)=\frac{\phi_1^{-1}(t)}{\Psi(\bfa,t)-\rho_0(\bfa)\sigma(t)}.
\end{equation}
To find $\phi_1$ note that for fixed $\bfa\in Q$ and $\mathbf{c}\in\mathbb{Z}^2$, the IVP
\begin{equation}
\label{flow0}
\mathbf{Z}'(t)=\bfu'(\mathbf{Z}(t),t),\qquad \mathbf{Z}(0)=\bfa+\mathbf{c}
\end{equation}
has a unique solution as long as $\bfu'=(u,v)$ stays smooth. Then by periodicity of $\bfu'$ and \eqref{dampedflowmap}, $\mathbf{Z}_1(t)=\bfX(\bfa,t)+\mathbf{c}$ and $\mathbf{Z}_2(t)=\bfX(\bfa+\mathbf{c},t)$ both solve \eqref{flow0} with the same initial condition. Thus $\bfX(\bfa,t)+\mathbf{c}=\bfX(\bfa+\mathbf{c},t)$, which implies that
\begin{equation}
\label{mean1}
\int_Q{J(\bfa,t)d\bfa}\equiv1.
\end{equation}
Integrating \eqref{dampedjacobianb0} now yields
\begin{equation}
\label{phib}
\phi_1(t)=\int_Q{\frac{d\bfa}{\Psi(\bfa,t)-\rho_0(\bfa)\sigma(t)}}.
\end{equation}
For $i\in\mathbb{Z}^+$, set  
\begin{equation*}
\label{kgamma}
\mathcal{K}_i(\bfa,t)=\frac{\gamma_0(\bfa)}{(\Psi(\bfa,t)-\rho_0(\bfa)\sigma(t))^i}\,,\qquad
\mathcal{\bar K}_i(t)=\int_Q{\frac{\gamma_0(\bfa)}{(\Psi(\bfa,t)-\rho_0(\bfa)\sigma(t))^i}d\bfa}
\end{equation*}
and 
\begin{equation*}
\label{krho}
\mathcal{L}_i(\bfa,t)=\frac{\rho_0(\bfa)}{(\Psi(\bfa,t)-\rho_0(\bfa)\sigma(t))^i}\,,\qquad
\mathcal{\bar L}_i(t)=\int_Q{\frac{\rho_0(\bfa)}{(\Psi(\bfa,t)-\rho_0(\bfa)\sigma(t))^i}d\bfa}.
\end{equation*}
Then using \eqref{rho}, \eqref{jacid}, \eqref{dampedjacobianb0} and \eqref{phib} we obtain, after a lengthy but straight-forward computation, the solution formula
\begin{equation}
\label{gammab}
\begin{split}
\gamma^B_{\alpha}(\bfX(\bfa,t),t)=\tau_{\alpha}'(t)\left(\mathcal{K}_1(\bfa,t)-\frac{\mathcal{\bar K}_2(t)}{\phi_1(t)}\right)+
\left(\mathcal{L}_1(\bfa,t)-\frac{\mathcal{\bar L}_2(t)}{\phi_1(t)}\right)\int_0^t{e^{\alpha s}\phi_1(s)ds}
\end{split}
\end{equation}
for $\tau_{\alpha}(t)$ satisfying 
\begin{equation}
\label{tauivpb}
\tau_{\alpha}'(t)=\frac{e^{-\alpha t}}{\phi_1^{2}(t)},\qquad \tau_{\alpha}(0)=0.
\end{equation}
Moreover, the Jacobian \eqref{dampedjacobianb0} becomes
\begin{equation}
\label{dampedjacobianb}
J(\bfa,t)=\frac{1}{\Psi(\bfa,t)-\rho_0(\bfa)\sigma(t)}\left(\int_Q{\frac{d\bfa}{\Psi(\bfa,t)-\rho_0(\bfa)\sigma(t)}}\right)^{-1}.
\end{equation}


Since \eqref{rho} implies that if $\rho_0\equiv0$, then $\rho_{\alpha}\equiv0$ for as long as the solution is defined, setting $\rho_0\equiv0$ in \eqref{gammab}-\eqref{tauivpb} yields, after some rearranging, the general solution of the damped Euler equation \eqref{dampedeqne} as
\begin{equation}
\label{gammae}
\begin{split}
\gamma^E_{\alpha}(\bfX(\bfa,t),t)=-\frac{\tau_{\alpha}'(t)}{\tau_{\alpha}(t)}\left\{\frac{1}{1+\gamma_0(\bfa)\tau_{\alpha}(t)}-\left(\int_Q{\frac{d\bfa}{1+\gamma_0(\bfa)\tau_{\alpha}(t)}}\right)^{-1}\int_Q{\frac{d\bfa}{(1+\gamma_0(\bfa)\tau_{\alpha}(t))^2}}\right\}
\end{split}
\end{equation}
for 
\begin{equation}
\label{tauivpe}
\tau_{\alpha}'(t)=e^{-\alpha t}\left(\int_Q{\frac{d\bfa}{1+\gamma_0(\bfa)\tau_{\alpha}(t)}}\right)^{-2},\qquad \tau_{\alpha}(0)=0.
\end{equation} 

\section{Proof of the Main Theorems}
\label{proofs}

The blowup result in Theorem \ref{blowupundamped} for the solution $\gamma^E$ of the undamped Euler equation \eqref{undampedgammaeqn} is established by estimating blowup rates for the integral terms in \eqref{undampedsln} under the assumption that the smooth initial data $\gamma_0$ behaves quadratically near the points where its minimum is attained \cite{constantin2}. Since we are interested in how damping can arrest this blowup, we consider the same class of initial data. In particular, this means that the blowup rates derived in \cite{constantin2} for the integral terms also hold here; however, and for convenience of the reader, we outline how to obtain these estimates in the proof of Theorem \ref{dampedeulerthm} below. 

\begin{proof}[\textbf{Proof of Theorem \ref{dampedeulerthm}}]

Suppose the mean-zero initial data  $\gamma_0(\bfa)$ is smooth and attains its negative minimum $m_0$ at a finite number of locations $\bfa_0\in Q$. Then the spatial term  
\begin{equation}
\label{spatial}
\frac{1}{1+\gamma_0(\bfa)\tau_{\alpha}(t)}
\end{equation}
in \eqref{gammae} diverges to positive infinity when $\bfa=\bfa_0$ as $\tau_{\alpha}$ approaches  $$\tau^*=-\frac{1}{m_0},$$
and remains finite and positive for all $0\leq\tau\leq\tau^*$ and $\bfa\neq\bfa_0$. Suppose $\gamma_0$ has nonzero second-order partials near $\bfa_0$, so that locally,
\begin{equation*}
\label{approx}
m_0+\frac{1}{2}\lambda_2|\bfa-\bfa_0|^2\leq \gamma_0(\bfa)\leq 
m_0+\frac{1}{2}\lambda_1|\bfa-\bfa_0|^2
\end{equation*}
for $0\leq|\bfa-\bfa_0|\leq r$, $r>0$ small, and $\lambda_1>\lambda_2>0$ the eigenvalues of the Hessian matrix of $\gamma_0$ at $\bfa_0$.\footnote[2]{If $\lambda_1=\lambda_2=\lambda$, then near $\bfa_0$,  $\gamma_0\sim m_0+\frac{1}{2}\lambda|\bfa-\bfa_0|^2$ and the blowup rates \eqref{int1}- \eqref{int2} still hold.} This implies that
\begin{equation}
\label{approx2}
C_1\ln\left(1+\frac{\lambda_1 r^2}{2\epsilon}\right)\leq
\int_{\mathcal{D}}\frac{d\bfa}{\epsilon+\gamma_0(\bfa)-m_0}
\leq
C_2\ln\left(1+\frac{\lambda_2 r^2}{2\epsilon}\right)
\end{equation}
for $\epsilon>0$ small, $C_i=\frac{2\pi}{\lambda_i}>0$, $i=1, 2$, and $\mathcal{D}$ the disk centered at $\bfa_0$ of radius $r$. Setting $\epsilon=\frac{1}{\tau_{\alpha}}+m_0$ in \eqref{approx2}, it follows that
\begin{equation}
\label{int1}
\int_{Q}\frac{d\bfa}{1+\gamma_0(\bfa)\tau_{\alpha}}\sim -C\ln(\tau^*-\tau_{\alpha})
\end{equation}
for $\tau^*-\tau_{\alpha}>0$ small. By a similar argument,
\begin{equation}
\label{int2}
\int_{Q}\frac{d\bfa}{(1+\gamma_0(\bfa)\tau_{\alpha})^2}\sim\frac{C}{\tau^*-\tau_{\alpha}}.
\end{equation}
Next we need the corresponding behavior of the exponential term
\begin{equation}
\label{exp}
e^{-\alpha t}=\operatorname{exp}\left\{-\alpha t_{\alpha}^E(\tau_{\alpha})\right\}
\end{equation}
in \eqref{gammae}-\eqref{tauivpe} as $\tau_{\alpha}\nearrow \tau^*$. In \eqref{exp} we have introduced the notation $t= t_{\alpha}^E$ to differentiate the time variable in \eqref{gammae}-\eqref{tauivpe} from that in \eqref{gammab}-\eqref{tauivpb}. For $\alpha>0$, we will refer to $t_{\alpha}^E$ and $t_{\alpha}^B$ as \emph{the damped Euler time variable} and respectively \emph{the damped Boussinesq time variable}.

Since $\gamma_0$ satisfies the conditions of Theorem \ref{blowupundamped}, the limit $T^E\equiv\lim_{\tau_{\alpha}\to\tau^*}t^E(\tau_{\alpha})$, for 
\begin{equation}
\label{undampedeulertime2}
t^E(\tau_{\alpha})=\int_0^{\tau_{\alpha}}\left(\int_{Q}{\frac{d\bfa}{1+\gamma_0(\bfa)\mu}}\right)^2d\mu
\end{equation}
the associated \emph{undamped Euler time variable} \eqref{undampedeulertime}, is positive and finite, and represents the blowup time for $\gamma^E$ in Theorem \ref{blowupundamped}. That $T^E$ is finite follows from \eqref{int1} and \eqref{undampedeulertime2} which imply, for $\tau^*-\tau_{\alpha}>0$ small, the asymptotic relation 
\begin{equation}
\label{asymp}
T^E-t^E\sim (\tau^*-\tau_{\alpha})(\ln(\tau^*-\tau_{\alpha})-1)^2
\end{equation}
whose right-hand side vanishes as $\tau_{\alpha}\nearrow\tau^*$. 

Since \eqref{tauivpe} implies that the damped Euler time variable satisfies
\begin{equation}
\label{time2}
t_{\alpha}^E(\tau_{\alpha})=-\frac{1}{\alpha}\ln\left|1-\alpha t^E(\tau_{\alpha})\right|,
\end{equation}
it follows that the behavior of the exponential \eqref{exp} is determined by the limit
\begin{equation}
\label{limit}
T_{\alpha}^E\equiv\lim_{\tau_{\alpha}\nearrow\tau^*}t_{\alpha}^E(\tau_{\alpha}),
\end{equation}
which in turn depends on whether $0<\alpha<1/T^E$ or $\alpha\geq1/T^E$.

\subsection*{Case 1 - Finite-time blowup for $0<\alpha<1/T^E$.}\hfill 

For $0<\alpha<1/T^E$, the argument of the logarithm in \eqref{time2} satisfies $0<1-\alpha t^E(\tau_{\alpha})\leq1$ for all $0\leq\tau_{\alpha}\leq\tau^*$. This implies that the limits
 \begin{equation*}
\label{case1time}
\lim_{\tau_{\alpha}\nearrow\tau^*}t_{\alpha}^E(\tau_{\alpha})=T_{\alpha}^E,\quad\qquad\lim_{\tau_{\alpha}\nearrow\tau^*}\operatorname{exp}\left\{-\alpha t_{\alpha}^E(\tau_{\alpha})\right\}=\operatorname{exp}\left\{-\alpha T_{\alpha}^E\right\}
\end{equation*}
are both positive and finite. Using estimates \eqref{int1} and \eqref{int2} on \eqref{gammae}, it follows that for $\bfa=\bfa_0$ the spatial term dominates and $\gamma_{\alpha}^E$ diverges to negative infinity,
$$\gamma_{\alpha}^E(\bfX(\bfa_0,t),t)\sim-\frac{C}{(\tau^*-\tau_{\alpha})\ln^2(\tau^*-\tau_{\alpha})}\to-\infty$$
as $t\nearrow T_{\alpha}^E$. If instead $\bfa\neq\bfa_0$, the second term in the bracket of \eqref{gammae} now dominates and the blowup is to positive infinity,
$$\gamma_{\alpha}^E(\bfX(\bfa,t),t)\sim-\frac{C}{(\tau^*-\tau_{\alpha})\ln^3(\tau^*-\tau_{\alpha})}\to+\infty.$$
From \eqref{time2} note that the blowup time $T_{\alpha}^E$ approaches the undamped blowup time $T^E$ as the damping coefficient vanishes. Further, for $\alpha>0$ and $\tau_{\alpha}\geq0$, we have that $t_{\alpha}^E(\tau_{\alpha})\geq t^E(\tau_{\alpha})$ with equality only at $\tau_{\alpha}=0$. Thus $T_{\alpha}^E>T^{E}$.
%
%

\subsection*{Case 2 - Convergence to a nontrivial steady state for $\alpha=1/T^E$.}\hfill

For $\alpha=1/T^E$, 
\begin{equation}
\label{case2time}
\lim_{\tau_{\alpha}\nearrow\tau^*}t_{\alpha}^E(\tau_{\alpha})=-T^E\lim_{\tau_{\alpha}\nearrow\tau^*}\left|1-\frac{t^E(\tau_{\alpha})}{T^E}\right|=+\infty
\end{equation}
and the exponential \eqref{exp} vanishes,
$$\lim_{\tau_{\alpha}\nearrow\tau^*}\operatorname{exp}\left\{-\alpha t_{\alpha}^E(\tau_{\alpha})\right\}=0.$$
To determine how fast we use \eqref{asymp} and \eqref{time2} to obtain
\begin{equation}
\label{exp2}
\operatorname{exp}\left\{-\alpha t_{\alpha}^E(\tau_{\alpha})\right\}
=\frac{1}{T^E}(T^E-t^E(\tau_{\alpha}))\sim C(\tau^*-\tau_{\alpha})(\ln(\tau^*-\tau_{\alpha})-1)^2
\end{equation}
for $\tau^*-\tau_{\alpha}>0$ small. Using \eqref{int1}-\eqref{int2} and  \eqref{case2time}-\eqref{exp2} on \eqref{gammae}, we see that $\gamma_{\alpha}^E$ converges to a mean-zero nontrivial steady state as $t\to+\infty$,
\begin{equation*}
\label{ae}
\lim_{t\to+\infty}\gamma_{\alpha}^{E}(\bfX(\bfa,t),t)=
\begin{cases}
-C,\qquad & \bfa=\bfa_0,
\\
0,\qquad & \bfa\neq\bfa_0.
\end{cases}
\end{equation*}

\subsection*{Case 3 - Convergence to the trivial steady state for $\alpha>1/T^E$.}\hfill

For $\alpha>1/T^E$, \eqref{time2} implies the existence of $0<\tau_1<\tau^*$ such that $t^E(\tau_{\alpha})\nearrow\frac{1}{\alpha}<T^E$ as $\tau_{\alpha}\nearrow\tau_1$. Then $t_{\alpha}^E(\tau_{\alpha})\to+\infty$ as $\tau_{\alpha}\nearrow\tau_1$, and
$$\lim_{\tau_{\alpha}\nearrow\tau_1}\operatorname{exp}\left\{-\alpha t_{\alpha}^E(\tau_{\alpha})\right\}=0.$$
Since the spatial and integral terms in \eqref{gammae} stay positive and finite for $0\leq\tau_{\alpha}\leq\tau_1$, it follows that $\gamma_{\alpha}^E(\bfx,t)\to0$ as $t\to+\infty$ for all $\bfx\in Q$.

\end{proof}


Before proving Theorem \ref{dampedboussinesqthm} we establish the following Lemma.

\begin{lemma}
\label{lemma}
Let $\rho_0(\bfx)\geq0$ and set $\tau^*=-1/m_0$. If $0\leq\tau_{\alpha}(t)<\tau^*$ for all $t\in\Xi\equiv[0,T)$, $0<T\leq+\infty$, then $0<\phi_1(t)<+\infty$ on $\Xi$.
\end{lemma}

\begin{proof}
Suppose $\rho_0(\bfa)\geq0$ and $0\leq\tau_{\alpha}(t)<\tau_*$ for all $t\in\Xi\equiv[0,T)$, $0<T\leq+\infty$. The latter and \eqref{psi} imply that $\Psi(\bfa,t)=1+\gamma_0(\bfa)\tau_{\alpha}(t)>0$ for all $(\bfa,t)\in Q\times\Xi$. Since $\phi_1(t)$ satisfies $\phi_1(0)>0$, there exists, by continuity, a positive $t_1\in\Xi$ such that $0<\phi_1(t)<+\infty$ for all $t\in[0,t_1)\subset\Xi$. Note that $\phi_1(t)$ cannot diverge to positive infinity at $t_1$. Indeed, suppose 
\begin{equation}
\label{phiinf}
\lim_{t\nearrow t_1}\phi_1(t)=+\infty.
\end{equation}
Since $\rho_0(\bfa)\geq0$, $\Psi(\bfa,t)>0$ for all $(\bfa,t)\in Q\times[0,t_1]$, and $\sigma(t)\leq0$ on $[0,t_1)$ (by \eqref{sigma}ii)), 
$$\Psi(\bfa,t)-\rho_0(\bfa)\sigma(t)\geq\Psi(\bfa,t)>0$$
for all $(\bfa,t)\in Q\times[0,t_1)$, and thus
\begin{equation}
\label{upsidedown}
\phi_1^{-1}(t)\geq\left(\int_Q{\frac{d\bfa}{\Psi(\bfa,t)}}\right)^{-1}>0
\end{equation}
for $t\in[0,t_1)$. From \eqref{phiinf} and \eqref{upsidedown} it follows that
$$\lim_{t\nearrow t_1}\int_Q{\frac{d\bfa}{\Psi(\bfa,t)}}=+\infty,$$
contradicting $\Psi(\bfa,t)>0$ for all $(\bfa,t)\in Q\times\Xi$. 

Now suppose there exists $t_2\in\Xi$ such that $\phi_1(t)$ vanishes as $t\nearrow t_2$, namely,
\begin{equation}
\label{phi0}
\lim_{t\nearrow t_2}\int_Q{\frac{d\bfa}{\Psi(\bfa,t)-\rho_0(\bfa)\sigma(t)}}=0.
\end{equation}
Since $\rho_0(\bfa)\geq0$ and $0<\phi_1(t)<+\infty$ on $[0,t_2)\subset\Xi$, $\sigma(t)\leq0$ is continuous on $[0,t_2)$. Then boundedness of $\tau_{\alpha}$ on $[0,t_2]$ and  $\Psi(\bfa,t)-\rho_0(\bfa)\sigma(t)>0$ for all $(\bfa,t)\in Q\times[0,t_2)$ imply that, for \eqref{phi0} to hold, $\sigma(t)\to-\infty$ as $t\nearrow t_2$. From this and \eqref{sigma}i) it follows that
\begin{equation*}
\label{a}
\lim_{t\nearrow t_2}\left\{\tau_{\alpha}(t)\int_0^t{e^{\alpha s}\phi_1(s)ds}\right\}=+\infty,
\end{equation*}
or since $0\leq\phi_1<+\infty$ on $[0,t_2]$, $\tau_{\alpha}(t)\to+\infty$ as $t\nearrow t_2$, contradicting $0\leq\tau_{\alpha}<\tau^*$ on $\Xi$.

\end{proof}

\begin{proof}[\textbf{Proof of Theorem \ref{dampedboussinesqthm}}]

Suppose $\gamma_0(\bfx)$ satisfies the conditions in Theorem \ref{blowupundamped} and $\rho_0(\bfx)\geq0$ for all $\bfx\in Q$. Further, and without loss of generality, assume there is one location $\bfa_1\in Q$ such that $\rho_0(\bfa_1)=0$ and $\gamma_0(\bfa_1)=m_0$. Set $\tau_*=-1/m_0$. Then Lemma \ref{lemma} and formulas \eqref{sigma} and \eqref{dampedjacobianb} imply that, as long as $0\leq\tau_{\alpha}<\tau_*$,
\begin{equation}
\label{jacest1}
J(\bfa,t)\geq\frac{1}{1+\gamma_0(\bfa)\tau_{\alpha}(t)-\rho_0(\bfa)\sigma(t)}\left(\int_Q{\frac{d\bfa}{1+\gamma_0(\bfa)\tau_{\alpha}(t)}}\right)^{-1}
\end{equation}
for all $\bfa\in Q$. Setting $\bfa=\bfa_1$ on \eqref{jacest1} and using \eqref{int1}, it follows that
\begin{equation*}
\label{jacest2}
J(\bfa_1,t)\geq\frac{1}{1+m_0\tau_{\alpha}(t)}\left(\int_Q{\frac{d\bfa}{1+\gamma_0(\bfa)\tau_{\alpha}(t)}}\right)^{-1}\sim-\frac{C}{(\tau^*-\tau_{\alpha})\ln(\tau^*-\tau_{\alpha})}
\end{equation*}
for $\tau^*-\tau_{\alpha}>0$ small, and therefore
\begin{equation*}
\label{jacest4}
J(\bfa_1,t)\to+\infty
\end{equation*}
as $\tau_{\alpha}\nearrow\tau^*$. Next, for $0<\alpha<1/T^E$ and $T^E>0$ the finite blowup time of the undamped Euler equation \eqref{undampedgammaeqn}, we prove the existence of a finite $T_{\alpha}^B>0$ such that $t\nearrow T_{\alpha}^B$ as  $\tau_{\alpha}\nearrow\tau^*$. From \eqref{phib}, \eqref{tauivpb} and Lemma \ref{lemma}, 
\begin{equation}
\label{timeb1}
e^{-\alpha t}dt\leq\left(\int_Q{\frac{d\bfa}{1+\gamma_0(\bfa)\tau_{\alpha}}}\right)^2d\tau_{\alpha}
\end{equation}
for $0\leq\tau_{\alpha}<\tau^*$. Denote the damped Boussinesq time variable $t=t_{\alpha}^B(\tau_{\alpha})$. Then integrating \eqref{timeb1} between 0 and $ t_{\alpha}^B$ yields the upper-bound 
\begin{equation}
\label{timeb2}
t_{\alpha}^B(\tau_{\alpha})\leq-\frac{1}{\alpha}\ln\left|1-\alpha t^E(\tau_{\alpha})\right|=t_{\alpha}^E(\tau_{\alpha})
\end{equation}
for $t_{\alpha}^B$ in terms of the damped Euler time variable $t_{\alpha}^E$, which in turn depends on the undamped Euler time variable $t^E$ in \eqref{undampedeulertime} and the damping coefficient $\alpha>0$. Since $\gamma_0$ satisfies the conditions of Theorem \ref{blowupundamped}, the limit $T^E\equiv\lim_{\tau_{\alpha}\nearrow\tau^*}t^E(\tau_{\alpha})$ is positive and finite. Suppose $0<\alpha<1/T^E$. Then Theorem \ref{dampedeulerthm} implies that $T_{\alpha}^E\equiv\lim_{\tau_{\alpha}\nearrow\tau^*}t_{\alpha}^E(\tau_{\alpha})$ is also positive and finite. Thus letting $\tau_{\alpha}\nearrow\tau^*$ in \eqref{timeb2}, we see that the blowup time $T_{\alpha}^B>0$ is finite and satisfies $$T_{\alpha}^B\equiv\lim_{\tau_{\alpha}\nearrow\tau^*}t_{\alpha}^B(\tau_{\alpha})\leq T_{\alpha}^E.$$
This finishes the proof of the first part of the Theorem. For the second part we adapt an argument used in \cite{childress, sarria2} to construct blowup and respectively global infinite-energy solutions of the 2D Euler and inviscid Boussinesq equations. 

Let $\rho_0(\bfx)=-\sin^2(2\pi x)$. Then we look for a solution of \eqref{ode} satisfying $\xi(\bfx,0)\equiv1$ and $\xi_t(\bfx,0)=\gamma_0(\bfx)$. Suppose $\mu_1(t)$ and $\mu_2(t)$ solve 
\begin{equation}
\label{mueqn}
\mu_1''+\alpha\mu_1'-I(t)\mu_1=0,\qquad \mu_2''+\alpha\mu_2'-I(t)\mu_2=1
\end{equation}
with $\mu_1(0)=\mu_1'(0)=1$ and $\mu_2(0)=0$, $\mu_2'(0)\neq0$. Then 
\begin{equation}
\label{mu}
\mu(\bfx,t)=\mu_1(t)+\rho_0(\bfx)\mu_2(t)
\end{equation}
satisfies
$$\mu_{tt}+\alpha\mu_t-I(t)\mu=\rho_0,\qquad \mu(\bfx,0)=1.$$
Note that $\mu_2'(0)\neq0$ must be such that $\gamma_0(\bfx)=\mu_t(\bfx,0)=1+\rho_0(\bfx)\mu_2'(0)$
has mean zero over $Q$, as required by \eqref{zeromean}. Now, since $J$ satisfies \eqref{mean1} it follows that 
$$1=\int_Q{\frac{d\bfx}{\mu_1-\sin^2(2\pi x)\mu_2}},$$
which yields the relation
\begin{equation}
\label{relation}
\mu_2=\mu_1-\frac{1}{\mu_1}.
\end{equation}
Using \eqref{relation} we see that $\gamma_0(\bfx)=\cos(4\pi x)$, which satisfies the mean-zero condition \eqref{zeromean}. Next, using \eqref{relation} on \eqref{mueqn}ii) to eliminate the nonlocal term $I(t)$ in \eqref{mueqn}i) gives, after some simplification,
$$\frac{d}{dt}\left(\frac{\mu_1'}{\mu_1}\right)+\alpha\frac{\mu_1'}{\mu_1}=\frac{\mu_1}{2}.$$
Dividing both sides by $\mu_1$, differentiating the resulting equation, and then setting $N(t)=\mu_1'/\mu_1$ now leads to
\begin{equation}
\label{i}
\frac{d}{dt}\left(N'-\frac{1}{2}N^2+\alpha N\right)=\alpha N^2.
\end{equation}
Since $N(0)=1$ and $N'(0)=\frac{1}{2}-\alpha$, we integrate \eqref{i} and use a standard Gronwall-type argument to obtain 
\begin{equation}
\label{z0}
z'\geq\frac{1}{2}e^{-\alpha t}z^2,\qquad z(0)=1
\end{equation}
for $z(t)=e^{\alpha t}N(t)$. Set 
\begin{equation}
\label{timeneg}
T_{\alpha}^B=-\frac{1}{\alpha}\ln\left(1-2\alpha\right)
\end{equation}
for $0<\alpha<1/2$. Note that $T_{\alpha}^B$ is positive and finite. Then solving \eqref{z0} for $t\in[0,T_{\alpha}^B)$ gives
$$z(t)\geq\frac{2\alpha}{e^{-\alpha t}-(1-2\alpha)},$$
or since $z(t)=e^{\alpha t}N(t)=e^{\alpha t}\frac{d}{dt}\left(\ln\mu_1(t)\right)$, 
\begin{equation}
\label{mu1}
\mu_1(t)\geq\frac{4\alpha^2}{\left(e^{-\alpha t}-(1-2\alpha)\right)^2}.
\end{equation}
From \eqref{mu1} it follows that $\mu_1\to+\infty$ as $t\nearrow T_{\alpha}^B$. Then by \eqref{mu} and  \eqref{relation},
$$\frac{1}{J(\bfx,t)}\geq\cos^2(2\pi x)\mu_1(t)+\frac{\sin^2(2\pi x)}{\mu_1(t)},$$
which implies that $J(\bfx,t)\to0$ as $t\nearrow T_{\alpha}^B$ for $\bfx\in\mathcal{B}\equiv\{(x,y)\in Q\,|\, x\notin\{1/4,3/4\}\}$. Note that $\mathcal{B}$ are precisely the points where $\gamma_0(\bfx)=\cos(4\pi x)$ does not equal its negative minimum.

Lastly, by setting $\alpha=0$ in \eqref{i} it is easy to see that the Jacobian of the associated undamped Boussinesq system vanishes as $t\nearrow 2$, which agrees with $T^B_{\alpha}\to2$ in \eqref{timeneg} as $\alpha\to0$. Thus we may write $0<\alpha<1/2$ as $0<\alpha<1/T^B$ for $T^B=2$ the blowup time for the solution of the undamped Boussinesq system with the same initial data.

\end{proof}

Note that the blowup result in part 2 of Theorem \ref{dampedboussinesqthm} is effectively one dimensional in the sense that the initial data depends only on one of the two coordinate variables. Currently we do not know if this particular blowup is suppressed for $\alpha\geq1/T^B$, or if a solution of the damped Boussinesq system \eqref{dampedeqnb}-\eqref{nonlocalb} persist globally in time for $\alpha\geq1/T^E$ when the initial data satisfies the conditions of part one of Theorem \ref{dampedboussinesqthm} \emph{and} is such that the associated undamped solution blows up in finite time.

\section{Acknowledgments}

This work was partially supported by the Div III \& P Research Funding Committee at Williams College.

\makeatletter \renewcommand{\@biblabel}[1]{\hfill#1.}\makeatother


\begin{thebibliography}{1000}


\bibitem[1]{adhikaria01} D. Adhikari, C. Cao and J. Wu, Global regularity results for the 2D Boussinesq equations with vertical dissipation, J. Differential Equations 251 (2011), 1637-1655.

\bibitem[2]{adhikaria}
D. Adhikari, C. Cao, J. Wu and X. Xu, Small global solutions to the damped two-dimensional Boussinesq equations, J. Differential Equations 256 (2014), 3594-3613.

\bibitem[3]{BKM} J.T. Beale, T. Kato, and A. Majda, Remarks on the breakdown of smooth solutions for the 3-D Euler equations, Commun. Math. Phys. \textbf{94} (1984), 61-66.

\bibitem[4]{chae07} D. Chae, Global regularity for the 2D Boussinesq equations with partial viscosity terms, Adv. Math. 203 (2006), 497-513.

\bibitem[5]{childress}
S. Childress, G.R. Ierley, E.A. Spiegel and W.R. Young, Blow-up of unsteady two-dimensional Euler and Navier-Stokes solutions having stagnation-point form, J. Fluid Mech. 203 (1989), 1-22.


\bibitem[6]{constantin2} P. Constantin, The Euler equations and nonlocal conservative Riccati equations, Int. Math. Res. Not. 9 (2000), 455-465.


\bibitem[7]{constantin1}
P. Constantin, Singular, weak, and absent: Solutions of the Euler equations, Physica D 237 (2008), 1926-1931.


\bibitem[8]{fefferman} C.L. Fefferman, \emph{Existence and smoothness of the Navier-Stokes equation}, the millennium prize problems, Clay Math. Inst., Cambridge, MA (2006), 57-67.

\bibitem[9]{gibbon1} J.D. Gibbon, A.S. Fokas, C.R. Doering, Dynamically stretched vortices as solutions of the 3D Navier–Stokes equations, Physica D (1999), 497-510.

\bibitem[10]{gibbon3} J.D. Gibbon, K. Ohkitani, Singularity formation in a class of stretched solutions of the equations for ideal magneto-hydrodynamics. Nonlinearity 14 (5) (2001), 1239-1264.

\bibitem[11]{gill} A.E. Gill, \textsl{Atmosphere-Ocean Dynamics}, Academic Press (London) (1982).


\bibitem[12]{hou1} T.Y. Hou, Blow–up or no blow-up? Unified computational and analytic approach to 3D incompressible Euler and Navier–Stokes equations, Acta Numerica (2009), 277-346.


\bibitem[13]{hou0} T.Y. Hou and C. Li, Global well-posedness of the viscous Boussinesq equations.
Discrete Contin. Dyn. Syst. 12 (2005), 1-12.

\bibitem[14]{hou22} T. Y. Hou and C. Li, Dynamic stability of the 3D axi-symmetric Navier-Stokes equations
with swirl, Comm. Pure Appl. Math., 61 (5) (2008), 661-697.

\bibitem[15]{lai} M.J. Lai, R.H. Pan, K. Zhao, Initial boundary value problem for 2D viscous Boussinesq equations, Arch. Ration. Mech. Anal., 199 (2011), 739-760.

\bibitem[16]{majda1} A. Majda, \emph{Introduction to PDEs and Waves for the Atmosphere and Ocean}, Courant Lecture Notes in Mathematics, vol. 9. AMS/CIMS, Providence (2003).

\bibitem[17]{majdabertozzi} A.J. Majda and A.L. Bertozzi, \emph{Vorticity and incompressible flow}, Cambridge Texts in Applied Mathematics, Cambridge, 2002.

\bibitem[18]{ohkitani} K. Ohkitani and J.D. Gibbon, Numerical study of singularity formation in a class
of Euler and Navier-Stokes flows. Physics of Fluids 12 (2000), 3181-3194.

\bibitem[19]{pedlosky} J. Pedlosky, \emph{Geophysical Fluid Dynamics}, Springer, New York (1987).

\bibitem[20]{sarria1} A. Sarria, Regularity of stagnation-point form solutions of the two-dimensional incompressible Euler equations, Differential and Integral Equations, 28 (3-4) (2015), 239-254.

\bibitem[21]{sarria2} A. Sarria and J. Wu, Blow-up in stagnation-point form solutions of the inviscid 2d Boussinesq equations, J Differ Equations, 259 (2015), 3559-3576.

\bibitem[22]{stuart} J.T. Stuart, Singularities in three-dimensional compressible Euler flows with vorticity, Theoret. Comput. Fluid Dyn. 10 (1998), 385-391.

\bibitem[23]{yang} W. Yang, Q. Jiu and J. Wu, Global well-posedness for a class of 2D Boussinesq systems with
fractional dissipation, J. Differential Equations 257 (2014), 4188-4213.

\end{thebibliography}
\end{document}